\documentclass[twoside,bezier,11pt, reqno]{amsart}
\usepackage{amsmath, amsthm, amscd, amsfonts,url, amssymb,graphicx,graphics, color}
\usepackage[bookmarksnumbered, plainpages]{hyperref}
\newtheorem{theorem}{Theorem}[section]

\newtheorem{corollary}[theorem]{Corollary}
\theoremstyle{definition}
\newtheorem{definition}[theorem]{Definition}
\newtheorem{example}[theorem]{Example}

\theoremstyle{remark}

\numberwithin{equation}{section}
\begin{document}
\begin{flushright}
{\bf\small Caspian Journal of Mathematical Sciences (CJMS)}\\
{\bf\small University of Mazandaran, Iran }\\
{\bf\small  \url{http://cjms.journals.umz.ac.ir}}\\
{\bf\small ISSN: 1735-0611}\\
\end{flushright}
{{\small CJMS}. {\bf xx}(x)(2020), xx-xx}\\
$\frac{}{\rule{5in}{0.04in}}$\\[.1in]
\vspace*{.5 cm}

{\bf\large WEAK TOPOLOGICAL CENTERS  AND COHOMOLOGICAL  PROPERTIES}
 \\[0.5cm]
{ Mostfa Shams Kojanaghi\footnote{ Corresponding author: haghnejad@uma.ac.ir\\ \qquad Received: 05 Month 2020\\ \qquad  Revised: 10 Month 2020 \\ \qquad Accepted: 10 Month 2020} and Kazem Haghnejad Azar $^2$\\
$^1$ Department of Mathematics, Ardabil Branch, Islamic Azad University, Ardabil, Iran
\email{\textcolor[rgb]{0.00,0.00,0.84}{mstafa.shams99@yahoo.com}}  \\
$^2$  Department of Mathematics, University of Mohaghegh Ardabili, Ardabil, Iran
\email{\textcolor[rgb]{0.00,0.00,0.84}{haghnejad@uma.ac.ir}}\\[2mm]

\vspace*{0.5cm}
%

\noindent
{\footnotesize {\sc Abstract.}

Let $B$ be a Banach $A-bimodule$. We introduce the weak topological centers of left  module action and we show it by 
$\tilde{{Z}}^\ell_{B^{**}}(A^{**})$.
For a compact group, we show that $L^1(G)=\tilde{Z}_{M(G)^{**}}^\ell(L^1(G)^{**})$ and on the other hand we have $\tilde{Z}_1^\ell{(c_0^{**})}\neq c_0^{**}$. Thus the weak topological centers  are different with topological centers of left or right  module actions.
In this manuscript, we investigate the relationships between  two concepts with some conclusions in Banach algebras.  We also have some application of this new concept and topological centers of module actions in the cohomological properties of Banach algebras, spacial, in the weak amenability and $n$-weak amenability of Banach algebras.\\



\markboth {K. Haghnejad Azar, M. Shams Kojanaghi }{Running Title}
\section{ Introduction}
\noindent 
\noindent Let $X,Y,Z$ be normed spaces and $m:X\times Y\rightarrow Z$ be a bounded bilinear mapping. Arens in \cite{1} offers two natural extensions $m^{***}$ and $m^{t***t}$ of $m$ from $X^{**}\times Y^{**}$ into $Z^{**}$ that he called $m$ is Arens regular whenever $m^{***}=m^{t***t}$, for more information see \cite{1, 13, 20}.

Recently, the subject of regularity of bounded bilinear mappings and Banach
module actions have been investigated in \cite{7, 10, 13, 13a}. In \cite{9},
Eshaghi Gordji and Fillali gave several significant results related to the
topological centers of Banach module actions. In \cite{20}, the authors have
obtained a criterion for the regularity of $f$, from which they gave several
results related to the regularity of Banach module actions with some
applications to the second adjoint of a derivation. For a good and rich
source of information on this subject, we refer the reader to the Memoir
in \cite{8}. We also shall mostly follow \cite{5} as a general reference on Banach
algebras.

Regarding $A$ as a Banach $A-bimodule$, the operation $\pi:A\times A\rightarrow A$ extends to $\pi^{***}$ and $\pi^{t***t}$ defined on $A^{**}\times A^{**}$. These extensions are known, respectively, as the first (left) and the second (right) Arens products, and with each of them, the second dual space $A^{**}$ becomes a Banach algebra. The regularity of a normed algebra $A$ is defined to be the regularity of its algebra multiplication when considered as a bilinear mapping.  Suppose that $A$ is a Banach algebra and $B$ is a Banach $A-bimodule$. Since $B^{**}$ is a Banach $A^{**}-bimodule$, where  $A^{**}$ is equipped with the first Arens product,  we  define the topological center of the   right module action of $A^{**}$ on $B^{**}$ as follows:
$${Z}^\ell_{A^{**}}(B^{**})={Z}(\pi_r)=\{b^{\prime\prime}\in B^{**}:~\text{the~map}~~a^{\prime\prime}\rightarrow \pi_r^{***}(b^{\prime\prime}, a^{\prime\prime})~:~A^{**}\rightarrow B^{**}$$$$~\text{is}~~~\text{weak}^*-\text{weak}^*~\text{continuous}\}.$$
In this way, we write ${Z}^\ell_{B^{**}}(A^{**})={Z}(\pi_\ell)$, ${Z}^r_{A^{**}}(B^{**})={Z}(\pi_\ell^t)$ and ${Z}^r_{B^{**}}(A^{**})={Z}(\pi_r^t)$, where $\pi_\ell:~A\times B\rightarrow B~~~and~~~\pi_r:~B\times A\rightarrow B$
are the  left and right  module actions of $A$ on $B$, for more information, see \cite{7, 10}.

If we set $B=A$, we write ${Z}^\ell_{A^{**}}(A^{**})=Z_1(A^{**})=Z^\ell_1(A^{**})$ and ${Z}^r_{A^{**}}(A^{**})=Z_2(A^{**})=Z_2^r(A^{**})$, for more information see \cite{ 18}.
Let $B$ be a  Banach $A-bimodule$  and  $n\geq 0$.  Suppose that $B^{(n)}$ is a $n-th~dual$ of  $B$. Then $B^{(n)}$ is also Banach $A-bimodule$, that is, for every $a\in A$, $b^{(n)}\in B^{(n)}$ and $b^{(n-1)}\in B^{(n-1)}$, we define
 $$\langle b^{(n)}a,b^{(n-1)}\rangle= \langle b^{(n)},ab^{(n-1)}\rangle,$$
 $$\langle ab^{(n)},b^{(n-1)}\rangle= \langle b^{(n)},b^{(n-1)}a\rangle.$$
Let $A^{(n)}$ and  $B^{(n)}$  be $n-th~dual$ of $A$ and $B$, respectively. By \cite{23},  for an even number $n\geq 0$, $B^{(n)}$ is a Banach $A^{(n)}-bimodule$. Then for $n\geq 2$,   we define  $B^{(n)}B^{(n-1)}$ as a subspace of $A^{(n-1)}$, that is, for all $b^{(n)}\in B^{(n)}$,  $b^{(n-1)}\in B^{(n-1)}$ and  $a^{(n-2)}\in A^{(n-2)}$ we define
$$\langle b^{(n)}b^{(n-1)},a^{(n-2)}\rangle=\langle b^{(n)},b^{(n-1)}a^{(n-2)}\rangle.$$
If $n$ is odd number, then for $n\geq 1$,   we define  $B^{(n)}B^{(n-1)}$ as a subspace of $A^{(n)}$, that is, for all $b^{(n)}\in B^{(n)}$,  $b^{(n-1)}\in B^{(n-1)}$ and  $a^{(n-1)}\in A^{(n-1)}$ we define
$$\langle b^{(n)}b^{(n-1)},a^{(n-1)}\rangle=\langle b^{(n)},b^{(n-1)}a^{(n-1)}\rangle.$$
and if $n=0$, we take $A^{(0)}=A$ and $B^{(0)}=B$.\\
So we can define the topological centers of module actions of $A^{(n)}$ on  $B^{(n)}$ similarly.


\section{\bf Weak topological center of module actions}

\noindent In this section, we introduce a new concept as  weak topological center of Banach algebras, module actions and we study their relationship  with topological centers of  module actions with some  conclusions  in the group algebras.

\begin{definition} Let $B$ be a Banach $A-bimodule$. We define the weak topological centers of left  module action as follows:
$$\tilde{{Z}}^\ell_{B^{**}}(A^{**})=\{ a^{\prime\prime}\in A^{**}:~\text{the~maping}~~~b^{\prime\prime}\rightarrow a^{\prime\prime}b^{\prime\prime}~\text{is}~\text{weak}^*-\text{weak~continuous}~\},$$
$$\tilde{{Z}}^\ell_{A^{**}}(B^{**})=\{ b^{\prime\prime}\in B^{**}:~\text{the~maping}~~~a^{\prime\prime}\rightarrow b^{\prime\prime}a^{\prime\prime}~\text{is}~~\text{weak}^*-\text{weak~continuous}~\}.$$
\end{definition}
Definition of $\tilde{ {Z}}^r_{B^{**}}(A^{**})$ and $\tilde{ {Z}}^r_{A^{**}}(B^{**})$ are similar for right module action. If $B=A$, we write $\tilde{{Z}}^\ell_{A^{**}}(A^{**})=\tilde{ Z}^\ell_1(A^{**})$ and  ${ \tilde{{Z}}^r_{A^{**}}}(A^{**})=\tilde{ Z^r_1}(A^{**})$, and the spaces $\tilde{ Z}^\ell_2(A^{**})$ and $\tilde{ {Z}^r_2}(A^{**})$ have similar definitions with respect to the second Arens product.
It is clear that $\tilde{Z}_i^\ell(A^{**})$ and $\tilde{{Z}_i^r}(A^{**})$ for each $i\in\{1,2\}$, are subspaces of $A^{**}$ with respect to the both Arens products. In general, by easy calculations, we have the following results:
\begin{enumerate}
 \item If $\tilde{Z}_1^\ell(A^{**})= A^{**}$ or $\tilde{{Z}_1^r}(A^{**})= A^{**}$, then $A$ is Arens regular.
 \item Let  $B\subseteq A^{**}$. Then  $B\tilde{Z}_1^\ell(A^{**})\subseteq \tilde{Z}_1^\ell(A^{**})$ and $\tilde{{Z}_1^r}(A^{**})B\subseteq \tilde{{Z}_1^r}(A^{**})$.
\item If $\tilde{Z}_1^\ell(A^{**})=A$, then $A$ is a right ideal in $A^{**}$.
\item If $\tilde{Z}_1^r(A^{**})=A$, then $A$ is a left ideal in $A^{**}$.
\item If $\tilde{Z}_1^\ell(A^{**})={Z}_1^r(A^{**})=A$, then $A$ is an ideal in $A^{**}$.
\item If $A^{***}A^{**}\subseteq A^*$, then $\tilde{Z}_1^\ell(A^{**})={{Z}_1}(A^{**})= A^{**}$.
\item If $A^{**}A^{***}\subseteq A^*$, then $\tilde{{Z}_1^r}(A^{**})= A^{**}$.
\item Suppose that  $A^{***}A\subseteq A^*$. If $A$ is  left strongly Arens irregular, then $\tilde{Z}_1^\ell(A^{**})=A$.
\end{enumerate}

 In the following example, we show that  if a Banach algebra $A$ is  Arens regular or strongly Arens irregular on the left, in general, $\tilde{Z}_1^\ell(A^{**})$ is not  $A^{**}$ or $A$,   respectively.

\begin{example}\label{2.2} 
\begin{enumerate}
 \item Let $A$ be nonreflexive Arens regular  Banach algebra and $e^{\prime\prime}\in A^{**}$ be a left unite element of $A^{**}$. Then $\tilde{Z}_1^\ell(A^{**})\neq A^{**}.$ Since $c_0^{**}=\ell^{\infty}$ and Arens product in $c_0^{**}$ coincide with the given natural product in $\ell^\infty$, $c_0^{**}$ is unital. Thus $\tilde{Z}_1^\ell{(c_0^{**})}\neq c_0^{**}$
 \item Suppose that  $G$ is  a locally compact group. Then by notice to \cite{22}, we know that in spacial case,  $M(G)$ is left strong Arens irregular, but $\tilde{Z}_1^\ell{(M(G)^{**}})\neq M(G)^{**}$.
 \item By (\cite{5}, Example 3.6.22(i)), we know that $c_0$ is Arens regular, and so $Z_1(c_0^{**})=c_0^{**}$, but we claim that $\tilde{Z}_1^\ell{(c_0^{**})}= c_0$. Indeed $c_0^{**}=\ell^{\infty}$ with the point-wise product. We can identity $\ell^{\infty}$ with $C(\beta \mathbb{N})$ the continuous functions on the Stone-Cech compactification, and then we find that $(\ell^{\infty})^*=M(\beta \mathbb{N})$ the measure space, and so the weak topology center is those $f\in C(\beta \mathbb{N})$ with $f\mu\in \ell^1$ for all $\mu\in M(\beta \mathbb{N})$. By considering point masses in $M(\beta \mathbb{N})$ (ultra-filter limits along $\mathbb{N}$), it is easy to show that $f\in c_0$.
 \end{enumerate}
\end{example}

\begin{definition} Let $A$ be a Banach algebra. The subspace of $A^{***}$ annihilating $A$ will be denoted by $A^\perp=\{a^{\prime\prime\prime}\in A^{***}:~a^{\prime\prime\prime}\mid_A=0\}.$
\end{definition}

\begin{theorem}\label{2.4} Let $A$ be a Banach algebra. Then we have the following assertions.
\begin{enumerate}
\item  If  $A$ is a left ideal in $A^{**}$, then  $A\subseteq\tilde{Z}_1^\ell(A^{**})\cap\tilde{Z}_2^\ell(A^{**})$.
\item If  $A$ is a right ideal in $A^{**}$, then  $A\subseteq\tilde{Z}_1^r(A^{**})\cap\tilde{Z}_2^r(A^{**})$.
\item If  $A$ is an ideal in $A^{**}$, then  $A\subseteq\tilde{Z}_1^\ell(A^{**})\cap \tilde{Z}_2^\ell(A^{**})\cap\tilde{Z}_1^r(A^{**})\cap\tilde{Z}_2^r(A^{**})$.
\item  If  $A$ is a left (resp. right) ideal in $A^{**}$ and $A^{**}$ has a left (resp. right) unit  in $\tilde{Z}_1^\ell(A^{**})$ (resp. $\tilde{Z}_1^r(A^{**})$), then $A$ is reflexive.
\end{enumerate}
\end{theorem}
\begin{proof}
\begin{enumerate}
\item  Assume that $(a_\alpha^{\prime\prime})_\alpha\subseteq A^{**}$ and
$a_\alpha^{\prime\prime}\stackrel{w^*} {\rightarrow}a^{\prime\prime}$. Let $a^{\prime\prime\prime}\in A^{***}$. Since $A^{***}=A^*\oplus A^\perp$, there are $a^\prime\in A^*$ and $t\in A^\perp$ such that $a^{\prime\prime\prime}=(a^\prime, t)$. Then for every $a\in A$, we have
$$\langle  a^{\prime\prime\prime},aa^{\prime\prime}_\alpha\rangle=\langle  (a^\prime, t),aa^{\prime\prime}_\alpha\rangle=\langle aa^{\prime\prime}_\alpha,a^\prime\rangle$$$$\rightarrow \langle aa^{\prime\prime},a^\prime\rangle=\langle  a^{\prime\prime\prime},aa^{\prime\prime}\rangle.$$
It follows that $A\subseteq\tilde{Z}_1^\ell(A^{**})$.  Since for every $a\in A$ and $a^{\prime\prime}\in A^{**}$, we have $aa^{\prime\prime}=aoa^{\prime\prime}$, similarly it follows that $A\subseteq\tilde{Z}_2^\ell(A^{**})$. Thus the result holds.
\item  The proof is similar to (1).
\item   Obvious.
\item Let $e\in \tilde{Z}_1^\ell(A^{**})$ be an unit element for $A^{**}$.  Set $a^{\prime\prime\prime}\in A^{***}$ and $(a_\alpha^{\prime\prime})_\alpha\subseteq A^{**}$ such that
$a_\alpha^{\prime\prime}\stackrel{w^*} {\rightarrow}a^{\prime\prime}$. Since $A^{***}=A^*\oplus A^\perp$, there are $a^\prime\in A^*$ and $t\in A^\perp$ such that $a^{\prime\prime\prime}=(a^\prime, t)$. Thus
$$\langle  a^{\prime\prime\prime},a^{\prime\prime}_\alpha\rangle=\langle  a^{\prime\prime\prime},ea^{\prime\prime}_\alpha\rangle=\langle  (a^\prime, t),ea^{\prime\prime}_\alpha\rangle=\langle ea^{\prime\prime}_\alpha,a^\prime\rangle$$$$\rightarrow \langle ea^{\prime\prime},a^\prime\rangle=\langle  a^{\prime\prime\prime},a^{\prime\prime}\rangle.$$
It follows that $a_\alpha^{\prime\prime}\stackrel{w} {\rightarrow}a^{\prime\prime}$. Hence  $A$ is reflexive.\end{enumerate}\end{proof}

\begin{corollary}\label{2.5}  Let $A$ be a Banach algebra. Then we have the following assertions.
\begin{enumerate}
\item  If  $A$ is a left ideal in $A^{**}$ and left strongly Arens irregular, then  $\tilde{Z}_1^\ell(A^{**})=A$, and so $A$ is an ideal in $A^{**}$.
\item If  $A$ is a right ideal in $A^{**}$ and  right strongly Arens irregular, then  $\tilde{Z}_1^r(A^{**})=A$, and so $A$ is an ideal in $A^{**}$.
\item If  $A$ is an ideal in $A^{**}$  and   strongly Arens irregular, then  $\tilde{Z}_1^\ell(A^{**})= \tilde{Z}_1^r(A^{**})=A$.
\item  If  $A$ is a left (resp. right) ideal in $A^{**}$ and $A^{**}$ has a left (resp. right) unit in $\tilde{Z}_1^\ell(A^{**})$ (resp. $\tilde{Z}_1^r(A^{**})$), then $A$ is reflexive.
\end{enumerate}
\end{corollary}
\begin{proof} By using Theorem \ref{2.4}, the proof holds.\end{proof}

\begin{example}\label{2.6} \begin{enumerate}
 \item  Let $G$ be a compact group. By using \cite{17}, we know that $L^1(G)$ is a strongly Arens irregular and  $L^1(G)$ is an ideal in its second dual.  Then by using the preceding corollary, we have $\tilde{Z}_1^\ell(L^1(G)^{**})=\tilde{Z}_1^r(L^1(G)^{**})=L^1(G)$.
\item  Let $G$ be a locally compact group. Then, in general, by the  preceding corollary, $M(G)$ is not a left or right  ideal in its second dual.
\end{enumerate}
\end{example}

\begin{theorem}\label{2.7} 
Let $B$ be a Banach $A-bimodule$. Then for every even number $n\geq 2$, we have the following assertions.
\begin{enumerate}
\item ${Z}^\ell_{A^{(n)}}(B^{(n+1)})=B^{(n+1)}$ if and only if ${\tilde{{Z}}^r}_{A^{(n)}}(B^{(n)})=B^{(n)}$.
\item   ${Z}^\ell_{A^{(n)}}(A^{(n+1)})=A^{(n+1)}$ if and only if $\tilde Z ^r_1(A^{(n)})=A^{(n)}$.
\item   ${Z}^r_{A^{(n)}}(A^{(n+1)})=A^{(n+1)}$ if and only if $\tilde Z^\ell_1(A^{(n)})=A^{(n)}$.
\item  ${Z}^r_{B^{(n)}}(A^{(n+1)})=A^{(n+1)}$ if and only if $\tilde{{Z}}^\ell_{A^{(n)}}(B^{(n)})=B^{(n)}$.
\end{enumerate}
\end{theorem}
\begin{proof} 1) Suppose that ${{{Z}^\ell}}_{A^{(n)}}(B^{(n+1)})=B^{(n+1)}$ and $b^{(n)}\in B^{(n)}$. We show that the mapping $a^{(n)}\rightarrow a^{(n)}b^{(n)}$ is $ weak^*-weak$ continuous. Assume that  $(a_\alpha^{(n)})_\alpha\subseteq A^{(n)}$ such that $a_\alpha^{(n)} \stackrel{w^*} {\rightarrow} a^{(n)}$. Then for all $b^{(n+1)}\in B^{(n+1)}$, we have $b^{(n+1)}a_\alpha^{(n)} \stackrel{w^*} {\rightarrow} b^{(n+1)}a^{(n)}$. It follows that
\begin{align*}
\langle  b^{(n+1)},a_\alpha^{(n)}b^{(n)}\rangle&=\langle  b^{(n+1)}a_\alpha^{(n)},b^{(n)}\rangle\\
&\rightarrow \langle  b^{(n+1)}a^{(n)},b^{(n)}\rangle\\
&=\langle  b^{(n+1)},a^{(n)}b^{(n)}\rangle.
\end{align*}
Thus we conclude that $b^{(n)}\in {\tilde{{Z}}^r}_{A^{(n)}}(B^{(n)})$.\\
The converse  is the same.\\
Proofs of (2), (3) and (4) are similar to (1).
\end{proof}

\begin{example}\label{2.8}  Let $A$ be a non-reflexive Banach space and let $\langle  f,x\rangle=1$ and $\|f\|\leq 1$ for some $f\in A^*$ and $x\in A$. We define the product on $A$ by $ab=\langle  f,b\rangle a$. It is clear that   $A$ is a Banach algebra with this product and it has  right identity $x$. By easy calculation, for all $a^\prime \in A^*$,  $a^{\prime\prime}\in A^{**}$ and $a^{\prime\prime\prime}\in A^{***}$, we have
\begin{align*}
a^\prime a&=\langle  a^\prime , a\rangle f,\\
a^{\prime\prime} a^\prime &=\langle  a^{\prime\prime}, f\rangle a^\prime,\\
a^{\prime\prime\prime} a^{\prime\prime} & =\langle  a^{\prime\prime}, a^{\prime\prime}\rangle\langle .,f\rangle.
\end{align*}
Therefore we have ${{{Z}^\ell}}_{A^{{**} }}(A^{{***}})\neq A^{{***}}$. So by Theorem \ref{2.7}, we have  ${\tilde{{Z}}^r}_{A^{{**}}}(A^{{**}})\neq A^{{**}}$.\\
Similarly, if we define the product on $A$ as $ab=\langle  f,a\rangle b$ for all $a, ~b\in A$, then we have ${{{Z}^\ell}}_{A^{{**} }}(A^{{***}})=A^{{***}}$. By using Theorem \ref{2.7}, it follows that ${\tilde{{Z}}^r}_{A^{{**}}}(A^{{**}})= A^{{**}}$.
\end{example}

\begin{theorem}\label{2.9}  Let $n>0$ be an even number and let  $B$ be a left (resp. right) Banach $A-module$ such that  $A^{(n-2)}B^{(n)}\subseteq B^{(n-2)}$ (resp. $B^{(n)}A^{(n-2)}\subseteq B^{(n-2)}$).
\begin{enumerate}
\item  Then $A^{(n-2)}\subseteq \tilde{{Z}}^\ell_{B^{(n)}}(A^{(n)})$ (resp. $A^{(n-2)}\subseteq \tilde{{Z}}^r_{B^{(n)}}(A^{(n)})$).
\item  If  $B^{(n)}$ has a left (resp. right) unit element in $\tilde{{Z}}^\ell_{B^{(n)}}(A^{(n)})$ (resp. $ \tilde{{Z}}^r_{B^{(n)}}(A^{(n)})$), then $A$ is reflexive.
\item  If $A^{(n-2)}\subset B^{(n-2)}$ and $A^{(n-2)}$ is left (resp. right) Arens irregular, then  $$A^{(n-2)}=\tilde{{Z}}^\ell_{B^{(n)}}(A^{(n)}) ~(resp.~A^{(n-2)}=\tilde{{Z}}^r_{B^{(n)}}(A^{(n)})).$$
\end{enumerate}
\end{theorem}
\begin{proof}
\begin{enumerate}
\item  Assume that $(b_\alpha^{(n)})_\alpha\subseteq B^{(n)}$ such that
$b_\alpha^{(n)}\stackrel{w^*} {\rightarrow}b^{(n)}$ in $B^{(n)}$. Let $b^{(n+1)}\in B^{(n+1)}$. Since $B^{(n+1)}=B^{(n-1)}\oplus B^\perp$, there are $b^{(n-1)}\in B^{(n-1)}$ and $t\in B^\perp$ such that $b^{(n+1)}=(b^{(n-1)}, t)$. Then for every $a^{(n-2)}\in A^{(n-2)}$, we have
\begin{align*}
\langle  b^{(n+1)},a^{(n-2)}b^{(n)}_\alpha\rangle
&=\langle  (b^{(n-1)}, t),a^{(n-2)}b^{(n)}_\alpha\rangle\\
&=\langle a^{(n-2)}b^{(n)}_\alpha,b^{(n-1)}\rangle\\
&\rightarrow \langle a^{(n-2)}b^{(n)},b^{(n-1)}\rangle\\
&=\langle  b^{(n+1)},a^{(n-2)}b^{(n)}\rangle.
\end{align*}
It follows that $A^{(n-2)}\subset\tilde{{Z}}^\ell_{B^{(n)}}(A^{(n)})$.
\item The  proof is clear.
\item  Since $A^{(n-2)}\subset B^{(n-2)}$,
$\tilde{{Z}}^\ell_{B^{(n)}}(A^{(n)})\subset Z_1(A^{(n)})=A^{(n-2)}.$ Thus by using part (1), since $\tilde{{Z}}^\ell_{B^{(n)}}(A^{(n)}))\subseteq {{Z}}^\ell_{B^{(n)}}(A^{(n)}))$, we are done.
 \end{enumerate}\end{proof}

\begin{example}\label{2.10}  Let $G$ be a compact group.  We know that  $L^1(G)\subseteq M(G)$ and  $L^1(G)$ is an ideal in $M(G)^{**}$.  Since  $L^1(G)$ is strongly Arens irregular, by preceding theorem,  we conclude that
 $$\tilde Z_{M(G)^{**}}^\ell(L^1(G)^{**})\subseteq Z_{M(G)^{**}}^\ell(L^1(G)^{**})\subseteq Z_1^\ell(L^1(G)^{**})=L^1(G).$$
By Theorem \ref{2.9}, we have  $L^1(G)\subseteq\tilde Z_{M(G)^{**}}^\ell(L^1(G)^{**})$. Thus we conclude that
 $$\tilde{Z}_{M(G)^{**}}^\ell(L^1(G)^{**})=L^1(G).$$
It is similar that
   $$\tilde{Z}_{M(G)^{**}}^r(L^1(G)^{**})=L^1(G).$$
\end{example}   
\section{\bf Weak amenability of Banach algebras}

Let $B$ be a   Banach $A-bimodule$. A derivation from $A$ into $B$ is a bounded linear mapping $D:A\rightarrow B$ such that $$D(xy)=xD(y)+D(x)y~~for~~all~~x,~y\in A.$$
The space of continuous derivations from $A$ into $B$ is denoted by $Z^1(A,B)$.
Easy example of derivations are the inner derivations, which are given for each $b\in B$ by
$$\delta_b(a)=ab-ba~~for~~all~~a\in A.$$
The space of inner derivations from $A$ into $B$ is denoted by $N^1(A,B)$.
The Banach algebra $A$ is said to be a amenable, when for every Banach $A-bimodule$ $B$, the inner derivations are only derivations existing from $A$ into $B^*$. It is clear that $A$ is amenable if and only if $H^1(A,B^*)=Z^1(A,B^*)/ N^1(A,B^*)=\{0\}$. The concept of amenability for a Banach algebra $A$, introduced by Johnson in 1972, has proved to be of enormous importance problems in Banach algebra theory, see \cite{14}. For Banach $A-bimodule$ $B$, the quotient space $H^1(A,B)$  is called the first cohomology group of $A$ with coefficients in $B$.\\
A Banach algebra $A$ is said to be a weakly amenable, if every derivation from $A$ into $A^*$ is inner. Similarly, $A$ is weakly amenable if and only if $H^1(A,A^*)=Z^1(A,A^*)/ N^1(A,A^*)=\{0\}$. The concept of weak amenability was first introduced by Bade, Curtis and Dales in \cite{2} for commutative Banach algebras, and was extended to the noncommutative case by Johnson in \cite{16}.
In every parts of this section, $n\geq 0$ is an even number.

\begin{theorem}\label{3.1}  Assume that $A$ is a Banach algebra and $\tilde{Z}_1^\ell(A^{(n)})=A^{(n)}$ where $n\geq 2$. If $A^{(n)}$ is weakly amenable, then $A^{(n-2)}$ is weakly amenable.
\end{theorem}
\begin{proof} Suppose that $D\in {{Z}}^1(A^{(n-2)},A^{(n-1)})$. First we show that $$D^{\prime\prime}\in {{Z}}^1(A^{(n)},A^{(n+1)}).$$ Let $a^{(n)},~ b^{(n)}\in A^{(n)}$  and let $(a_\alpha^{(n-2)})_\alpha,~ (b_\beta^{(n-2)})_\beta\subseteq A^{(n-2)}$ such that $a_\alpha^{(n-2)}\stackrel{w^*} {\rightarrow}a^{(n)}$ and
$b_\beta^{(n-2)}\stackrel{w^*} {\rightarrow}b^{(n)}$. Since ${{Z}_1^\ell}(A^{(n)})=A^{(n)}$ ~we have ~ $$\lim_\alpha \lim_\beta a_\alpha^{(n-2)}D(b_\beta^{(n-2)})=a^{(n)}D^{\prime\prime}( b^{(n)}).$$
On the other hand, we have
  $$\lim_\alpha \lim_\beta D(a_\alpha^{(n-2)})b_\beta^{(n-2)}=D^{\prime\prime}( a^{(n)})b^{(n)}.$$ Since $D$ is continuous, we conclude that
\begin{align*}
D^{\prime\prime}( a^{(n)}b^{(n)})
&=\lim_\alpha \lim_\beta D(a_\alpha^{(n-2)}b_\beta^{(n-2)})\\
&=\lim_\alpha \lim_\beta a_\alpha^{(n-2)}D(b_\beta^{(n-2)})+\lim_\alpha \lim_\beta D(a_\alpha^{(n-2)})b_\beta^{(n-2)}\\
&=a^{(n)}D^{\prime\prime}(b^{(n)})+D^{\prime\prime}( a^{(n)})b^{(n)}.
\end{align*} 
 In the above equalities, the convergence are with respect to  $weak^*$ topology.
Since $A^{(n)}$ is weakly amenable, $D^{\prime\prime}$ is inner. It follows  that $D^{\prime\prime}( a^{(n)})=
a^{(n)}a^{(n+1)}-a^{(n+1)}a^{(n)}$ for some $a^{(n+1)}\in A^{(n+1)}$. Set  $a^{(n-1)}=a^{(n+1)}\mid_{A^{(n-2)}}$ and\\
$a^{(n-2)}\in A^{(n-2)}$. Then $$D(a^{(n-2)})=D^{\prime\prime}(a^{(n-2)})=a^{(n-2)}a^{(n-1)}-a^{(n-1)}a^{(n-2)}=\delta_{a^{(n-1)}}(a^{(n-2)}).$$
Consequently, we have  $H^1(A^{(n-2)},A^{(n-1)})=0$, and so $A^{(n-2)}$ is weakly amenable.\end{proof}

\begin{corollary}\label{3.2}  Let $A$ be a Banach algebra and let $\widetilde{wap_\ell}(A^{(n-1)})\subseteq A^{(n)}$  whenever $n\geq1$. If  $A^{(n)}$ is weakly amenable, then  $A^{(n-2)}$ is weakly amenable.
\end{corollary}
\begin{proof} Since $\widetilde{wap_\ell}(A^{(n-1)})\subseteq A^{(n)}$, $\tilde{Z}_1^\ell(A^{(n)})=A^{(n)}$. Then,  by using Theorem
\ref{3.1}, the proof holds.\end{proof}

\begin{corollary}\label{3.3}  Let $A$ be a Banach algebra  and ${{{Z}^\ell}}_{A^{(n)}}(A^{(n+1)})=A^{(n+1)}$, where $n\geq 2$. If  ${A^{(n)}}$ is weakly amenable, then ${A^{(n-2)}}$ is weakly amenable.
\end{corollary}

\begin{corollary}\label{3.4}  Let $A$ be a Banach algebra and let $D:A^{(n-2)}\rightarrow A^{(n-1)}$ be a derivation where $n\geq 2$. Then $D^{\prime\prime}:A^{(n)}\rightarrow A^{(n+1)}$ is a derivation when $\tilde{Z}_1^\ell(A^{(n)})=A^{(n)}$.
\end{corollary}

\begin{theorem}\label{3.5}  Let $A$ be a Banach algebra and let $B$ be a closed subalgebra of $A^{(n)}$ that is consisting of $A^{(n-2)}$ where $n\geq 2$. If $\tilde{Z}_1^\ell(B)=B$ and   $B$ is weakly amenable, then $A^{(n-2)}$ is weakly amenable.
\end{theorem}
\begin{proof} Suppose that $D:A^{(n-2)}\rightarrow A^{(n-1)}$ is a derivation and $p:A^{(n+1)}\rightarrow B^*$ is the restriction map, defined by $P(a^{(n+1)})=a^{(n+1)}\mid_B$ for every $a^{(n+1)}\in A^{(n+1)}$. Since $\tilde{Z}_1^\ell(B)=B$, $\bar{D}=PoD^{\prime\prime}\mid_B:B\rightarrow B^\prime$ is a derivation. Since $B$ is weakly amenable, there is $b^\prime\in B^*$ such that $\bar{D}=\delta_{b^\prime}$. We take $b^{(n-1)}=b^\prime\mid_{A^{(n-2)}}$, then $D=\bar{D}$ on $A^{(n-2)}$. Consequently, we have $D=\delta_{b^{(n-1)}}$.\end{proof}

\begin{corollary}\label{3.6}  Let $A$ be a Banach algebra. If $A^{***}A^{**}\subseteq A^{*}$ and $A$ is weakly amenable, then $A^{**}$ is weakly amenable.
\end{corollary}
\begin{proof} By using Corollary \ref{2.4} and Theorem \ref{3.2}, proof  holds.\end{proof}

\begin{corollary}\label{3.7}  Let $A$ be a Banach algebra and let $\tilde{Z}_1^\ell(A^{(n)})$ be weakly amenable
whenever $n\geq 2$. Then  $A^{(n-2)}$ is weakly amenable.
\end{corollary}

\begin{theorem}\label{3.8} Let $B$ be a   Banach $A-bimodule$ and $D:A^{(n)}\rightarrow B^{(n+1)}$ be a derivation for $n\geq 0$. If $\tilde{{Z}^\ell}_{A^{(n+2)}}(B^{(n+2)})=B^{(n+2)}$, then $D^{\prime\prime}:A^{(n+2)}\rightarrow B^{(n+3)}$ is a derivation.
\end{theorem}
\begin{proof}
Let $x^{(n+2)},~ y^{(n+2)}\in A^{(n+2)}$  and let $(x_\alpha^{(n)})_\alpha,~ (y_\beta^{(n)})_\beta\subseteq A^{(n)}$ such that $x_\alpha^{(n)}\stackrel{w^*} {\rightarrow}x^{(n+2)}$ and
$y_\beta^{(n)}\stackrel{w^*} {\rightarrow}y^{(n+2)}$ in $A^{(n+2)}$. Then for all $b^{(n+2)}\in B^{(n+2)}$, we have
 $b^{(n+2)}x_\alpha^{(n)}\stackrel{w} {\rightarrow}b^{(n+2)}x^{(n+2)}$. Consequently, since $\tilde{{Z}^\ell}_{A^{(n+2)}}(B^{(n+2)})=B^{(n+2)}$, we have
\begin{align*} 
\langle  x_\alpha^{(n)}D^{\prime\prime}(y^{(n+2)}),b^{(n+2)}\rangle
&=\langle  D^{\prime\prime}(y^{(n+2)}),b^{(n+2)}x_\alpha^{(n)}\rangle\\
&\rightarrow \langle  D^{\prime\prime}(y^{(n+2)}),b^{(n+2)}x^{(n+2)}\rangle\\
&=\langle  x^{(n+2)}D^{\prime\prime}(y^{(n+2)}),b^{(n+2)}\rangle.
\end{align*}
Also we have the following equality
\begin{align*}
\langle  D^{\prime\prime}(x^{(n+2)})y_\beta^{(n)},b^{(n+2)}\rangle
&=\langle  D^{\prime\prime}(x^{(n+2)}),y_\beta^{(n)}b^{(n+2)}\rangle\\
&\rightarrow\langle  D^{\prime\prime}(x^{(n+2)}),y^{(n+2)}b^{(n+2)}\rangle\\
&=\langle  D^{\prime\prime}(x^{(n+2)})y^{(n+2)},b^{(n+2)}\rangle.
\end{align*}
Since $D$ is continuous, it follows that
\begin{align*}
D^{\prime\prime}(x^{(n+2)}y^{(n+2)})
&=\lim_\alpha \lim_\beta D(x_\alpha^{(n)}y_\beta^{(n)})\\
&=\lim_\alpha \lim_\beta x_\alpha^{(n)}D(y_\beta^{(n)})+\lim_\alpha \lim_\beta D(x_\alpha^{(n)})y_\beta^{(n)}\\
&=x^{(n+2)}D^{\prime\prime}(y^{(n+2)})+D^{\prime\prime}(x^{(n+2)})y^{(n+2)}.
\end{align*}
\end{proof}

\begin{corollary}\label{3.9}  Let $B$ be  a   Banach $A-bimodule$ and $\tilde{Z}^\ell_{A^{**}}(B^{{**}})=B^{{**}}$. If  $H^1(A,B^*)={0}$, then
$H^1(A^{**},B^{***})={0}$.
\end{corollary}

\begin{corollary}\label{3.10}  Let $B$ be a   Banach $A-bimodule$ and $\tilde{{Z}^\ell}_{A^{{**}}}(B^{{**}})=B^{{**}}$. If $D:A\rightarrow B^*$ is a derivation, then $D^{\prime\prime}(A^{**})B^{**}\subseteq A^*$.
\end{corollary}
\begin{proof} By using Theorem \ref{3.8}, Corollary \ref{3.3}  and \cite{20} proof  holds.
\end{proof}

\section{ \bf Cohomological properties of Banach algebras  }
 Let $A$ be a Banach algebra and $n\geq 0$. Then $A$ is called $n-weakly$ amenable if $H^1(A,A^{(n)})=0$,  and   is called permanently weakly amenable when $A$ is $n-weakly$ amenable for each $n\geq 0$. In \cite{6} Dales, Ghahramani, and Gronbaek introduced the concept of n-weak amenability for Banach algebras for each natural number $n$. They determined some
relations between m- and n-weak amenability for general Banach algebras
and for Banach algebras in various classes, and proved that, for every  $n$,
(n + 2)- weak amenability always implies n-weak amenability.

\vspace{0.2cm}

\begin{theorem}\label{4.1}  Let $B$ be a Banach $A-bimodule$ and let $n\geq 1$. If $H^1(A,B^{(n+2)})=0$, then $H^1(A,B^{(n)})=0$.
\end{theorem}
\begin{proof} Let $D\in Z^1(A,B^{(n)})$ and  $i: B^{(n)}\rightarrow B^{(n+2)}$ be the canonical linear mapping as $A-bimodule$ homomorphism. Take $\widetilde{D}=ioD$. Then we can be viewed $\widetilde{D}$ as an element of
$Z^1(A,B^{(n+2)})$. Since $H^1(A,B^{(n+2)})=0$, there exist $b^{(n+2)}\in B^{(n+2)}$ such that  $$\widetilde{D}(a)=ab^{(n+2)}-b^{(n+2)}a,$$
for all $a\in A$. Set an $A-linear$ mapping $P$ from $B^{(n+2)}$ into $B^{(n)}$ such that $Poi=I_{B^{(n)}}$. Then we have $Po\widetilde{D}=(Poi)oD=D$, and so   $D(a)=Po\widetilde{D}(a)=aP(b^{(n+2)})-P(b^{(n+2)})a$ for all $a\in A$. It follows that $D\in N^1(A,B^{(n)})$. Consequently  $H^1(A,B^{(n)})=0$.
\end{proof}

\begin{theorem}\label{4.2}  Let $B$ be a Banach $A-bimodule$ and  $D:A\rightarrow B^{(2n)}$ be a continuous derivation. Assume that $Z^\ell_{A^{(2n)}}(B^{(2n)})=B^{(2n)}$. Then there is a continuous derivation $\widetilde{D}: A^{(2n)}\rightarrow B^{(2n)}$ such that $\widetilde{D}(a)=D(a)$ for all $a\in A$.
\end{theorem}
\begin{proof} By using  Proposition 1.7 from \cite{6}, the linear mapping $D^{\prime\prime}:A^{**}\rightarrow B^{(2n+2)}$ is a continuous derivation. Take $X=B^{(2n-2)}$. Since $Z_{A^{(2n)}}(X^{**})=Z_{A^{(2n)}}(B^{(2n)})=B^{(2n)}=X^{**}$, by Proposition 1.8 from  \cite{6} the canonical projection $P:  X^{(4)}\rightarrow X^{**}$ is an $A^{**}-bimodule $ morphism. Set $\widetilde{D}=PoD^{\prime\prime}$. Then $\widetilde{D}$ is a continuous derivation from $A^{**}$ into $B^{(2n)}$. Now by replacing $A^{**}$ by $A$ and repeating of the proof, result holds.  \end{proof}

\vspace{0.2cm}

\begin{corollary}\label{4.3} Let $B$ be a Banach $A-bimodule$ and $n\geq 0$. If $Z^\ell_{A^{(2n)}}(B^{(2n)})=B^{(2n)}$ and $H^1(A^{(2n+2)},B^{(2n+2)})=0$, then  $H^1(A,B^{(2n)})=0$.
\end{corollary}
\begin{proof} By using Proposition 1.7 from \cite{6} and preceding theorem the result  holds. \end{proof}

\begin{corollary}\label{4.4} \cite{6}. Let $A$ be a Banach algebra such that $A^{(2n)}$ is Arens regular and
 $H^1(A^{(2n+2)}),A^{(2n+2)})=0$ for each $n\geq 0$. Then $A$ is $2n-weakly$ amenable.
\end{corollary}

Assume that $A$ is Banach algebra and $n\geq 0$. We define $A^{[n]}$ as a subset of $A$ as follows
$$A^{[n]}=\{a_1a_2...a_n:~a_1,a_2,...a_n\in A\}.$$
We write $A^n$  the linear span of $A^{[n]}$ as a subalgebra of $A$.

\begin{theorem}\label{4.5} Let $A$ be a Banach algebra and $n\geq 0$. Let $A^{[2n]}$ dense in $A$ and suppose that $B$ is a Banach $A-bimodule$. Assume that $AB^{**}$ and $B^{**}A$ are subsets of $B$. If $H^1(A,B^*)=0$, then  $H^1(A ,B^{(2n+1)})=0$.
\end{theorem}
\begin{proof} For $n=0$ the result is clear.  Let $B^\perp$ be the space of functionals in $B^{(2n+1)}$ which annihilate $i(B)$ where $i:B\rightarrow
B^{(2n)}$ is a natural canonical mapping. Then, by using lemma 1  \cite{23}, we have the following equality
$$B^{(2n+1)}=i(B)^*\oplus B^\bot ,$$
it follows that
$$H^1(A,B^{(2n+1)})=H^1(A,i(B)^*)\oplus H^1(A,B^\bot ) .$$ Without lose generality, we replace $i(B)^*$ by $B^*$.
Since $H^1(A,B^*)=0$, it is enough to show that $H^1(A,B^\bot )=0$.\\
Now, take the linear mappings $L_a$ and $R_a$ from $B$ into itself by $L_a(b)=ab$ and $R_a(b)=ba$ for all  $a\in A$.
Since $AB^{**}\subseteq B$ and $B^{**}A\subseteq B$,  $L^{**}_a(b^{\prime\prime})=ab^{\prime\prime}$ and $R^{**}_a(b^{\prime\prime})=b^{\prime\prime}a$ for every $a\in A$, respectively. Consequently, $L_a$ and $R_a$ from $B$ into itself are weakly compact. It follows that for each $a\in A$ the linear mappings $L^{(2n)}_a$ and   $R^{(2n)}_a$ from $B^{(n)}$ into $B^{(n)}$ are weakly compact and for every $b^{(2n)}\in B^{(2n)}$, we have $L^{(2n)}_a(b^{(2n)})=ab^{(2n)}\in B^{(2n-2)}$ and $R^{(2n)}_a(b^{(2n)})=b^{(2n)}a\in B^{(2n-2)}$. Set $a_1, a_2, ...,a_n\in A$ and $b^{(2n)}\in B^{(2n)}$. Then $a_1 a_2 ...a_nb^{(2n)}$ and $b^{(2n)}a_1 a_2 ...a_n$ are belong to $B$.
Suppose that $D\in Z^1(A,B^\bot )$ and let $a, x\in A^{[n]}$. Then for every $b^{(2n)}\in B^{(2n)}$, since $xb^{(2n)}, b^{(2n)}a \in B$, we have the following equality
$$\langle D(ax), b^{(2n)}\rangle=\langle aD(x), b^{(2n)}\rangle+\langle D(a)x, b^{(2n)}\rangle$$$$=\langle D(x), b^{(2n)}a\rangle+\langle D(a), xb^{(2n)}\rangle=0.$$
It follows that $D\mid_{A^{[2n]}}=0$. Since $A^{[2n]}$ dense in $A$, $D=0$. Hence $H^1(A,B^\bot )=0$ and result follows.
\end{proof}

\begin{corollary}\label{4.6} \begin{enumerate}
\item  Let $A$ be a Banach algebra with left bounded approximate identity, and let  $B$ be a Banach $A-bimodule$.  Suppose that $AB^{**}$ and $B^{**}A$ are subset of $B$. Then  $H^1(A,B^{(2n+1)})=0$ for all $n\geq 0$, whenever $H^1(A,B^*)=0$.
\item  Let $A$ be an amenable Banach algebra and $B$ be a Banach $A-bimodule$. If $AB^{**}$ and $B^{**}A$ are subset of $B$, then $H^1(A,B^{(2n+1)})=0$.
\end{enumerate}
\end{corollary}

\begin{example} \label{4.7} Assume that $G$ is a compact group.
\begin{enumerate}
 \item We know that $L^1(G)$ is $M(G)-bimodule$  and $L^1(G)$ is an ideal in the second dual of $M(G)$, $M(G)^{**}$. By using corollary 1.2 from \cite{19}, we have $H^1(L^1(G),M(G)^*)=0$. Then  for every $n\geq 1$, by using preceding corollary, we conclude that
$$H^1(L^1(G),M(G)^{(2n+1)})=0.$$
\item Since $L^1(G)$ is an ideal in its second dual , $L^1(G)^{**}$, by using \cite{16},  $L^1(G)$ is a weakly amenable. Then by preceding corollary, $L^1(G)$ is $(2n+1)-weakly$  amenable.
\end{enumerate}
\end{example}

\begin{corollary}\label{4.8} Let $A$ be a Banach algebra and let $A^{[2n]}$ be dense in $A$. Suppose that $AB^{**}$ and $B^{**}A$ are subset of $B$. Then the following are equivalent.
\begin{enumerate}
\item ~$H^1(A,B^*)=0$.
\item ~$H^1(A,B^{(2n+1)})=0$ for some $n\geq 0$.
\item ~~$H^1(A,B^{(2n+1)})=0$ for each $n\geq 0$.
\end{enumerate}
\end{corollary}

\begin{corollary}\label{4.9} \cite{6}. Let $A$ be a weakly amenable Banach algebra such that $A$ is an ideal in $A^{**}$. Then $A$ is $(2n+1)-weakly$ amenable for each $n\geq 0$.
\end{corollary}
\begin{proof} By using Proposition 1.3 from \cite{6} and preceding theorem, result holds. \end{proof}

Assume that $A$ and $B$ are Banach algebras. Then $A\oplus B$ , with norm $$\parallel(a,b)\parallel=\parallel a\parallel+\parallel b\parallel,$$ and product $(a_1,b_1)(a_2,b_2)=(a_1a_2, b_1b_2)$ is a Banach algebra. It is clear that if $X$ is a Banach $A~and ~B-bimodule$, then $X$ is a Banach $A\oplus B-bimodule$.\\
In the following, we investigated the relationships between the cohomological property of  $A\oplus B$ with $A$ and $B$.

\begin{theorem}\label{4.10} Suppose that $A$ and $B$ are Banach algebras. Let $X$ be a Banach $A~and ~B-bimodule$. Then,   $H^1(A\oplus B, X)=0$ if and only if $H^1(A,X)=H^1(B,X)=0$.
\end{theorem}
\begin{proof} Suppose that $H^1(A\oplus B, X)=0$. Assume that $D_1\in Z^1(A,X)$ and $D_2\in Z^1(B,X)$. Take $D=(D_1,D_2)$. Then for every $a_1,a_2\in A$ and $b_1, b_2\in B$, we have
$$D((a_1,b_1)(a_2,b_2))=D(a_1a_2,b_1b_2)=(D_1(a_1a_2),D_2(b_1b_2))$$
$$=(a_1D_1(a_2)+D_1(a_1)a_2,b_1D_2(b_2)+D_2(b_1)b_2)$$
$$=(a_1D_1(a_2),b_1D_2(b_2))+(D_1(a_1)a_2+D_2(b_1)b_2)$$
$$=(a_1,b_1)(D_1(a_2),D_2(b_2))+(D_1(a_1),D_2(b_1))(a_2,b_2)$$
$$=(a_1,b_1)D(a_2,b_2)+D(a_1,b_1)(a_2,b_2).$$
It follows that $D\in Z^1(A\oplus B,X)$.
Since $H^1(A\oplus B, X)=0$, there is $x\in X$ such that $D=\delta_{x}$  where  $\delta_{x}\in N^1(A\oplus B,X)$. Since
$\delta_{x}=(\delta^1_{x},\delta^2_{x})$ where $\delta^1_{x}\in N^1(A,X)$ and $\delta^2_{x}\in N^1(B,X)$,  we have $D_1=\delta^1_{x}$ and $D_2=\delta^2_{x}$. Thus $H^1(A,X)=H^1(B,X)=0$.\\
For the converse, take $A$ as an ideal in $A\oplus B$, and so by using Proposition 2.8.66 from \cite{5}, proof holds. \end{proof}

\begin{example} \label{4.11} Let $G$ be a locally compact group and $X$ be a Banach $L^1(G)-bimodule$. Then by \cite{6}, pp.27 and 28, $X^{**}$ is a Banach $L^1(G)^{**}-bimodule$.  Since $L^1(G)^{**}=LUC(G)^*\oplus LUC(G)^\bot$, by using preceding theorem, we have  $$H^1(L^1(G)^{**},X^{**})=0$$ if and only if $H^1(LUC(G)^*,X^{**})=H^1(LUC(G)^\bot,X^{**})=0$.\\
On the other hand, we know that $L^1(G)^{**}=L^1(G)\oplus C_0(G)^\perp$. By \cite{16}, we know that, $H^1(L^1(G), L^\infty(G))=0$. Then by using preceding theorem,  $H^1(L^1(G)^{**}, L^\infty(G))=0,$ if and only if  $H^1(C_0(G)^\perp, L^\infty(G))=0$.
\end{example}


\end{document}